\documentclass[12pt]{amsart}

\usepackage{amssymb,latexsym}

\usepackage{color}

\usepackage[T1]{fontenc}

\usepackage{enumerate}

 \usepackage[french,english]{babel}
 
\usepackage[utf8]{inputenc}
\usepackage{geometry}
\usepackage{amsmath}
\usepackage{amssymb}
\usepackage{amsfonts}
\usepackage{graphicx}
\usepackage{appendix}
\usepackage{tabularx}
\usepackage{booktabs}
\usepackage{amsthm}
\usepackage{tabu}
\usepackage{bbm}
\usepackage{multirow}
\usepackage{enumitem}
\usepackage{mathtools}
\usepackage{dsfont}
\usepackage{esint}
\usepackage{longtable}

\makeatletter

\@namedef{subjclassname@2010}{

  \textup{2010} Mathematics Subject Classification}

\makeatother
\newtheorem{thm}{Theorem}[section]

\newtheorem{cor}[thm]{Corollary}
\newtheorem{lem}[thm]{Lemma}
\newtheorem{pro}[thm]{Proposition}
\theoremstyle{definition}

\newtheorem{rem}[thm]{Remark}

\numberwithin{equation}{section}

\renewcommand{\epsilon}{\varepsilon}

\newcommand{\legendre}[2]{\genfrac{(}{)}{}{}{#1}{#2}}

\newtheorem{theorem}{Theorem}[section]

\newcommand{\X}{\mathbb{X}}
\newcommand{\Y}{\mathbb{Y}}

\newcommand{\ex}{\mathbb{E}}

\newcommand{\pr}{\mathbb{P}}
\newcommand{\ep}{\varepsilon}

\newcommand{\B}{\mathcal{B}}

\newcommand{\un}{\mathbf{n}}
\newcommand{\ut}{\mathbf{t}}

\newcommand{\newabstract}[1]{%
  \par\bigskip
  \csname otherlanguage*\endcsname{#1}%
  \csname captions#1\endcsname
  \item[\hskip\labelsep\scshape\abstractname.]
}

\begin{document}

\baselineskip=17pt

\title{The limiting distribution of Legendre paths}

\author{Ayesha Hussain}

\address{Department of Mathematics,
University of Exeter,
North Park Road,
Exeter,
EX4 4QF,
U.K.}
\email{a.s.hussain@exeter.ac.uk}

\author{Youness Lamzouri}

\address{Universit\'e de Lorraine, CNRS, IECL,  and  Institut Universitaire de France,
F-54000 Nancy, France}

\email{youness.lamzouri@univ-lorraine.fr}


\begin{abstract} Let $p$ be a prime number and 
$\left(\frac{\cdot}{p}\right)$ be the Legendre symbol modulo $p$. The \emph{Legendre path} attached to $p$ is the polygonal path whose vertices are the normalized character sums $\frac{1}{\sqrt{p}} \sum_{n\leq j} \left(\frac{n}{p}\right)$ for $0\leq j\leq p-1$. In this paper, we investigate the distribution of Legendre paths as we vary over the primes $Q\leq p\leq 2Q$, when $Q$ is large. Our main result shows that as $Q \to \infty$, these paths converge in law, in the space of real-valued continuous functions on $[0, 1]$, to a certain random Fourier series constructed using Rademacher random completely  multiplicative functions. This was previously proved by the first author under the assumption of the Generalized Riemann Hypothesis. 

\end{abstract}

\subjclass[2010]{Primary 11L40 11N64; Secondary 11K65}

\thanks{}

\maketitle

\section{Introduction}\label{s: introduction_real_characters}
    
    A central question in number theory is to understand the behavior and the size of the character sum
     \begin{align}\label{eq: definition of character sum}
        S_p(x) := 
        \sum_{n \leq x} \legendre{n}{p},
    \end{align}
    where $\left(\frac{\cdot}{p}\right)$ is the Legendre symbol modulo an odd prime $p$ and $1\leq x\leq p$.
    Such sums encode important
information on the distribution of quadratic residues and non-residues modulo $p$. In particular, bounds for $S_p(x)$ lead to results on the size of the least quadratic non-residue
modulo $p$ (see for example the works of Ankeny \cite{An}; 
Burgess \cite{Bu}; Graham and Ringrose \cite{GrRi}; and Montgomery
\cite{Mo}). Moreover, a recent work of Granville and Mangerel \cite{GrMa} shows that cancellations in the sum $S_p(x)$ in short initial intervals are more or less equivalent to improved bounds for the value of the Dirichlet $L$-function $L\big(s, \big(\frac{\cdot}{p}\big)\big)$ at $s=1$, which is connected to class numbers of the quadratic extensions $\mathbb{Q}(\sqrt{\pm p})$.  

The Legendre symbol appears to exhibit random behavior over short intervals. This is for example illustrated by the work of Davenport and Erd\"{o}s \cite{DaEr}, who studied the distribution of $S_p(x)$ over short moving intervals. More precisely, they showed that if $p$ is large and $x$ varies among the integers $\{0, 1, \dots, p-1\}$, the short sum $S_p(x+H)-S_p(x)$ tends to a normal distribution with mean zero and variance $H$, provided $\log H = o(\log p)$ and $H \rightarrow \infty$ as $p \rightarrow \infty$. Recently, Harper \cite{Ha3} showed that this is no longer the case if $H$ is much larger, namely when $H\geq p/(\log p)^A$, and $A>0$ is a fixed constant. 

Our goal in this paper is to gain some understanding on the distribution of $S_p(x)$ over long intervals. In particular we will attempt to answer the following question: do these character sums, viewed as functions of $x$, possess a limiting distribution as $p$ varies? and if so, can we describe this distribution?

\noindent Asked this way, this question does not make sense since one has $\max_{1\leq x\leq p}|S_p(x)|\gg \sqrt{p}$ (this is an easy consequence of Parseval's Theorem applied to P\'olya's Fourier expansion \eqref{eq: definition Legendre path} below), and $\max_{1\leq x\leq p}|S_p(x)|\ll_{\ep} \sqrt{p}$ for all but at most $O(\ep \pi(Q))$ primes $p\leq Q$ (this was proved by Montgomery and Vaughan \cite{MoVa79}), where $\pi(Q)$ is the number of primes $p\leq Q$. Therefore, we need to normalize the sum $S_p(x)$ by $\sqrt{p}$.  
Moreover, since $S_p(x)$ is periodic with period $p$, we shall normalize the length of the sum by $p$, and consider the function $t\to \frac{1}{\sqrt{p}} S_p(tp)$.  However, a final complication remains: this function is discontinuous, with jumps at every $t \in \frac{1}{p}\mathbb{N}$. Instead, we shall consider the continuous function $f_p$ on $[0,1]$, where we concatenate the points where $\frac{1}{\sqrt{p}} S_p(tp)$ changes. We call such a function the  \textbf{Legendre path} associated to $p$. Note that for $t\in [0,1]$ we have
\begin{equation}\label{Eq: RelationFpSp}
f_p(t) = \frac{1}{\sqrt{p}}S_p(tp) + \frac{\{ pt \}}{\sqrt{p}} \legendre{\lceil pt \rceil }{p}= \frac{1}{\sqrt{p}}S_p(tp) +  O\left(\frac{1}{\sqrt{p}}\right),
\end{equation}  
where $\{x \}$ is the fractional part of $x$.

In this paper, we shall prove that these paths, viewed as random processes on $[0, 1]$, have a nice limiting distribution. Before stating our result, we shall construct a probabilistic random model for the Legendre path $f_p$.  
A good random model for the Legendre symbol, which was extensively studied in recent years, is a \emph{Rademacher random completely multiplicative function}. This is defined as
\begin{equation}\label{eq: RademacherRandom}\X_n:= \prod_{q^k|| n} \X_q^k
\end{equation}
for $n\in \mathbb{N}$, where $\X_1=1$, and $\{\X_q\}_{q \text{ prime}}$ is a sequence of I. I. D. random variables taking the values $\pm 1$ with probability $1/2$ each. The distribution of sums of Rademacher random multiplicative functions was extensively studied in recent years, in particular by Harper \cite{Ha1} and \cite{Ha2}.  
In studying the distribution of the character sum $S_p(x)$, our first guess is to model the normalized sum $\frac{1}{\sqrt{p}}S_p(tp)$  by the following normalized sum of random variables 
$$ \frac{1}{\sqrt{p}}\sum_{n\leq tp }\X_n.$$
However, it turns out that this is not a good model for $\frac{1}{\sqrt{p}}S_p(tp)$, since it does not account for the periodicity of the Legendre symbol modulo $p$. To exploit this periodicity, we shall use the Fourier series expansion of $f_p$, which was first established by P\'olya in the following quantitative form
    (eq. (9.19), p. 311 of \cite{MoVaBook})
    \begin{equation}\label{eq: definition Legendre path}
        f_p(t) = \frac{\tau(\legendre{\cdot}{p})}{2 \pi i \sqrt{p}} \sum_{1 \leq  |n| \leq  Z} \legendre{n}{p} \frac{1 - e(-nt)}{n} + O\left(\frac{1}{\sqrt{p}}+\frac{\sqrt{p}\log p}{Z}\right),
    \end{equation}
    for any real number $Z\geq 1$, where $e(t):= e^{2\pi i t}$ and $\tau(\legendre{\cdot}{p})$ is the Gauss sum associated to the Legendre symbol modulo $p$, more precisely
    \begin{equation}\label{eq: GaussSum}
    \begin{aligned}
        \tau\left(\legendre{\cdot}{p}\right) := \sum_{a = 1}^p \legendre{a}{p} e(a/p) = \begin{dcases}
             \sqrt{p}, & \text{ if } p \equiv 1 \pmod 4,  \\
             i \sqrt{p}, & \text{ if } p \equiv 3 \pmod 4.
        \end{dcases}
    \end{aligned}
    \end{equation}

    \begin{rem}\label{Rem: EvenOddLegendre}
    The value of the prime modulo $4$ influences the shape of the Legendre path, as shown in Figure \ref{fig: characterpaths}. In fact, a simple calculation implies that for all $j\in \{0, 1, \dots, p-1\}$ we have $S_p(p-1-j)=S_p(j)$ if $p\equiv 3\pmod 4$, and $S_p(p-1-j)=-S_p(j)$ if $p\equiv 1\pmod 4$. This shows that if we extend the Legendre path $f_p$ to $\mathbb{R}$ by periodicity, the periodic extension is an even function if $p\equiv 3\pmod 4$, and is an odd function if $p\equiv 1\pmod 4$.  

    \end{rem}
    
    \begin{figure} 
        \centering
        \begin{tabular}{cc}
            \includegraphics[width=6cm]{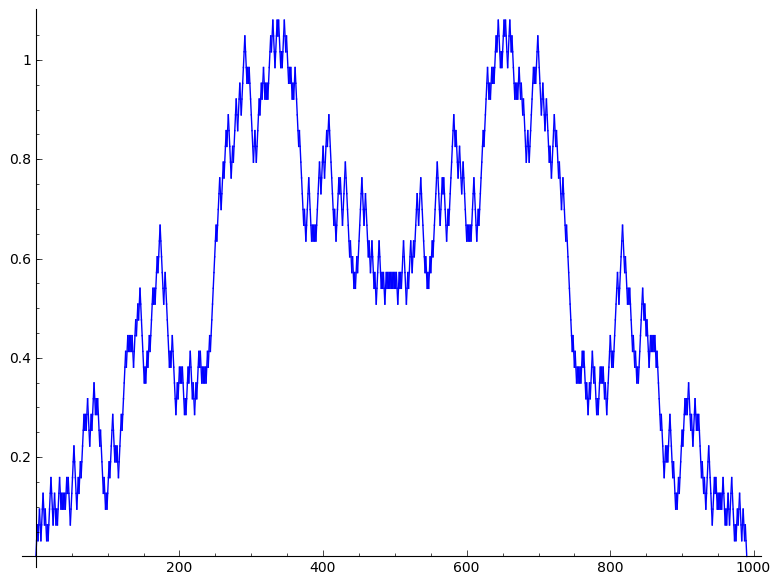} &\includegraphics[width=6cm]{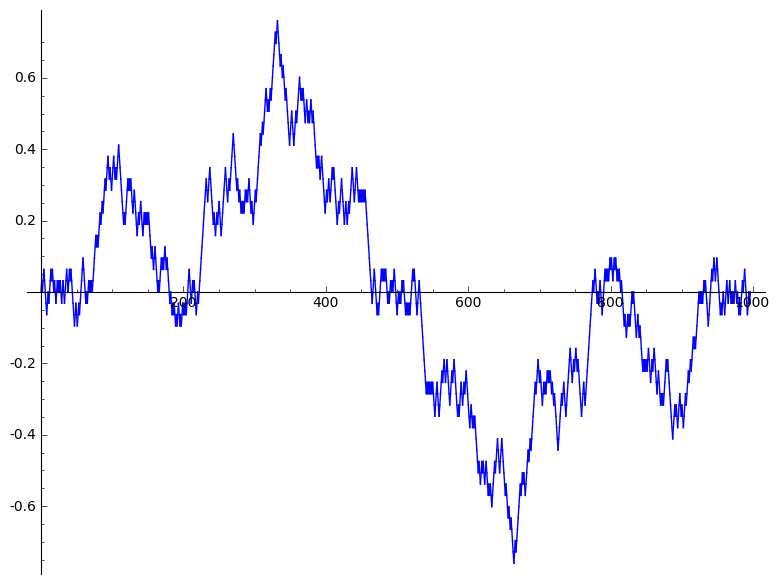} \\
            \small{The Legendre path for $p = 991 \equiv 3 \pmod 4$ } &
            \small{The Legendre path for $p = 997 \equiv 1 \pmod 4$.}
        \end{tabular}
        \caption{Legendre paths for $p = 991$ and $p = 997$, where the $x$-axis is $pt$.}
        \label{fig: characterpaths}
    \end{figure}
    Inserting the values of the Gauss sum \eqref{eq: GaussSum} in the Fourier expansion \eqref{eq: definition Legendre path} gives
    \begin{equation}\label{eq: PolyaFourier}
f_p(t) = \frac{\varepsilon_p}{2 \pi i } \sum_{1 \leq  |n| \leq  Z} \legendre{n}{p} \frac{1 - e(-nt)}{n} + O\left(\frac{1}{\sqrt{p}}+\frac{\sqrt{p}\log p}{Z}\right),
    \end{equation}
    where $\varepsilon_p= 1$ if $p\equiv 1\pmod 4$, and equals $i$ if $p\equiv 3\pmod 4$. Let $\{\X_q\}_{q \text{ prime}}$ be a sequence of I. I. D. random variables taking the values $\pm 1$ with equal probability $1/2$. Since the primes split evenly into the residue classes $1$ and $3 \pmod 4$, we shall model the value $\left(\frac{-1}{p}\right)$ by a random variable $\X_{-1}$, which is independent from the $\X_q$'s for $q$ prime, and takes the values $\pm1$ with equal probability $1/2$. We also extend the definition of Rademacher random completely multiplicative functions to the negative integers using multiplicativity, namely by setting $\X_{-n}=\X_{-1}\X_n$ for $n\in \mathbb{N}$. Combining all these elements together leads us to model the Legendre path $f_p$ by the following random Fourier series
$$ F_{\X}(t):= \frac{\Y}{2 \pi i} \sum_{n \in \mathbb{Z}\setminus\{0\}} \X_n \frac{1 - e(-nt)}{n} $$
for $t\in [0,1]$, where the random variable $\Y$ is defined by $\Y= 1$ if $\X_{-1}=1$, and $\Y=i$ if $\X_{-1}=-1$. For a fixed $t\in [0, 1]$, it follows from Lemma 1 of Kalmynin \cite{Ka} that  the series defining $F_{\X}(t)$ converges almost surely as the limit of the partial sums 
$$F_{\X, N}(t):=\frac{\Y}{2 \pi i} \sum_{0<|n|\leq N} \X_n \frac{1 - e(-nt)}{n}. $$
Note that $F_{\X, N}$ are $C([0,1])$-valued random variables. We will show, in Proposition \ref{pro: ContinuityRandom } below, that the sequence of processes $(F_{\X, N})_{N\geq 1}$ converges in $C([0, 1])$ to the process $F_{\X}$ as $N\to \infty$. As a consequence, we deduce that $F_{\X}$ is almost surely the Fourier series of a continuous function. Figure \ref{fig: random paths} shows samples of the random process $F_{\X}$ depending on the value of $\X_{-1}.$

Let $Q$ be a large positive integer. We shall view each $p \mapsto f_p$ for $Q\leq p\leq 2Q$ as a random variable on the finite probability space $\{Q\leq p\leq 2Q\}$ endowed with the uniform probability measure. We denote this random variable by $\mathcal{F}_Q$. Our main result shows that the stochastic process $(\mathcal{F}_Q)_{Q}$ converges in law, in the space of continuous functions $C([0, 1])$, to the random process $F_{\X}$.

    \begin{figure} 
        \centering
        \begin{tabular}{cc}
            \includegraphics[width=6cm]{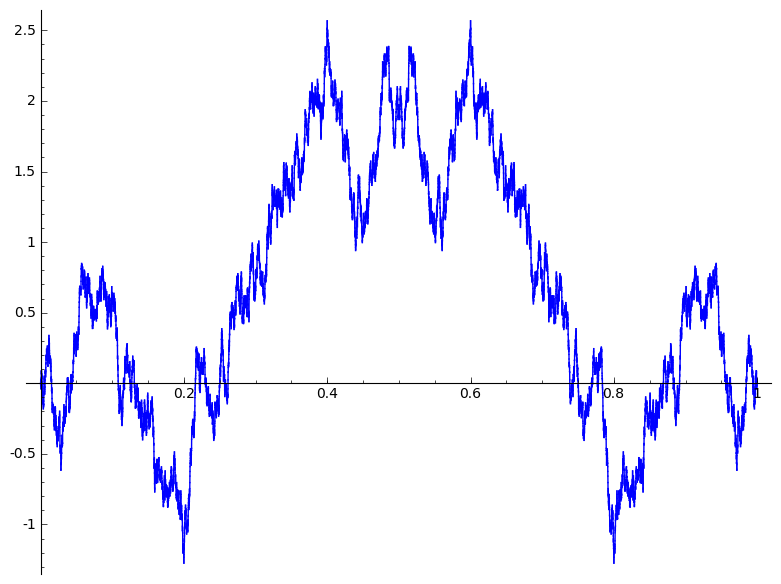} &\includegraphics[width=6cm]{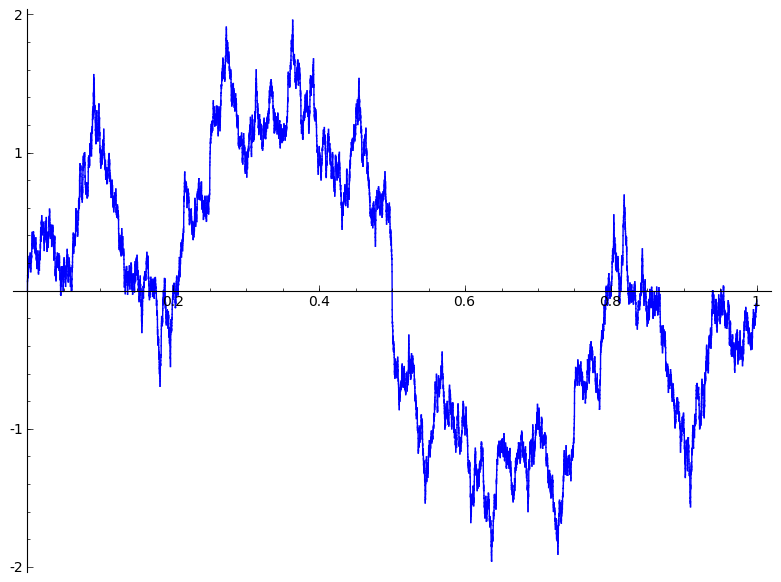} \\
            \small{Sample of $F_{\X}(t)$ with $\X_{-1} = -1$.} &
            \small{Sample of $F_{\X}(t)$ with $\X_{-1} = 1$}
        \end{tabular}
        \caption{Samples, with 10,000 points, of $F_{\X}(t)$, where the $x$ axis is $t$.}
        \label{fig: random paths}
    \end{figure}

    \begin{theorem}\label{thm: main theorem}
As $Q\to \infty$, the sequence of processes $(\mathcal{F}_Q)_{Q}$ converges to the process $F_{\X}$
  in the
sense of convergence in law in the Banach space $C([0, 1])$ endowed with the topology of uniform convergence. 
 More precisely, for any continuous and bounded map
        \begin{align*}
            \phi: C([0,1]) \rightarrow \mathbb{R},
        \end{align*}
        we have
        \begin{align*}
            \lim_{Q \rightarrow \infty} \mathbb{E}(\phi(\mathcal{F}_{Q})) = \mathbb{E}(\phi(F_{\X})),
        \end{align*}
        where 
        $$ \mathbb{E}(\phi(\mathcal{F}_{Q})):=\frac{1}{\pi^*(Q)}\sum_{Q\leq p\leq 2Q} \phi(f_p),$$
        and $\pi^*(Q)$ is the number of primes $p$ in the interval $[Q, 2Q]$.
    \end{theorem}
    
    The proof of this result is split into two main parts. We first show in Section \ref{s: proof of convergence in finite distributions real} that the sequence of processes $(\mathcal{F}_Q)_Q$ converges to the random process $F_{\X}$ in the sense of finite distributions, as $Q\to \infty$. By Prokhorov's Theorem (see for example \cite[Theorem 5.1]{Bi}) convergence in law in $C([0,1])$ is then equivalent to showing that the sequence of processes $(\mathcal{F}_Q)_Q$ is \emph{tight}, which we establish in Section \ref{s: Tighness}.

    \begin{rem} The ``graphs'' of various exponential sums were investigated by several authors since the seventies. See for example the works of Lehmer \cite{Le} and Loxton \cite{Lo1} and \cite{Lo2}. 
    Our work is motivated by the recent papers of Kowalski and Sawin \cite{KoSa} on Kloosterman paths of prime moduli (paths constructed using partial sums of Kloosterman sums), and Ricotta-Royer \cite{RiRo}, Ricotta-Royer-Shparlinski \cite{RiRoSh}, and Milicevich-Zhang \cite{MiZh} on generalisations to Kloosterman paths of prime power moduli. However, unlike these works, where the corresponding equidistribution problem can be seen as ``vertical'', ours can be viewed as a ``horizental'' problem, which is usually believed to be more challenging. 

   In \cite{Hu} the first author investigated the distribution of character paths (paths constructed using partial sums of non-principal Dirichlet characters modulo a prime $q$), and proved the analogue of Theorem \ref{thm: main theorem} for these paths. In this case, the sequence of corresponding processes  converges in law to an analogous random Fourier series as $q\to \infty$, where the Fourier coefficients are Steinhaus random multiplicative functions, and where the random variable $\Y$ is uniformly distributed on the unit circle. We should note that since the Legendre symbol is a real character, the Legendre path is a time graph, while the Kloosterman paths and character paths (for non-real Dirichlet characters) are instead polygonal curves in the complex plane.  
      \end{rem}   
        
        In her thesis, the first author proved Theorem \ref{thm: main theorem} assuming the Generalized Riemann Hypothesis GRH. Our unconditional proof is completely different and uses several new ideas to remove GRH 
 from the different parts of the proof. More precisely, to prove that the sequence of processes $(\mathcal{F}_Q)_Q$ converges in law to the random process $F_{\X}$ in the sense of finite distributions, we compute in Section 3 the multivariate moments of $(f_{p}(t_1), \dots, f_p(t_k))$ (for $0\le t_1<\dots <t_k\leq 1$) unconditionally, by using a recent result of the second author \cite{La} to considerably shorten the P\'olya Fourier expansion \eqref{eq: definition Legendre path} of $f_p$ for almost all $p$. We also need to bound the contribution of the possible exceptional discriminant.
 At the end of Section 3, we describe an alternative approach to establishing convergence in the sense of finite distributions, which is discussed in Section 6.2 of  Kowalski's book \cite{Ko}, and relies instead on showing that the sequence of ``Fourier coefficients'' of the Legendre path converges in the sense of finite distributions to those of $F_{\X}$. Although this approach is somewhat simpler, we decided to include both since the original argument gives a quantitative asymptotic formula for the joint moments of the Legendre path $f_p$ at different points, which is interesting on its own, and might have further applications. Furthermore, in Section 4 we prove tightness of the sequence of processes $(\mathcal{F}_Q)_Q$ by appealing to a variant of Kolmogorov's Tightness Criterion from \cite{Ko}, and then using Burgess's bound for short character sums along with the quadratic large sieve inequalities of Heath-Brown \cite{HB} and Montgomery and Vaughan \cite{MoVaBook}, in order to bound several moments of $|f_p(t)-f_p(s)|$ depending on the range of $|t-s| \in [0,1]$. We should note that our proof is very different from the work of the first author \cite{Hu} on character paths associated to the family of non-principal characters modulo a large prime $q$, since this family behave very differently from that of Legendre symbols. 
 Indeed, the crucial ingredients used in \cite{Hu} 
 are the orthogonality relations for characters modulo $q$, as well as Deligne's bound for hyper-Kloosterman sums, which is used by the first author to control certain twisted moments of Gauss sums.

    \begin{rem} Recall that the behavior of the Legendre path attached to $p$ depends on the value of the Legendre symbol $\left(\frac{-1}{p}\right)$. Indeed, it follows from \eqref{eq: PolyaFourier} that for any real number $Z\geq 1$ we have 
    \begin{equation} \label{eq: Polya1mod4}f_p(t) = \frac{1}{\pi } \sum_{1 \leq  |n| \leq  Z} \legendre{n}{p} \frac{\sin(2\pi n t)}{n} + O\left(\frac{1}{\sqrt{p}}+\frac{\sqrt{p}\log p}{Z}\right),
    \end{equation}
    if $p\equiv 1 \pmod 4$, and 
  \begin{equation} \label{eq: Polya3mod4}f_p(t) = \frac{1}{\pi } \sum_{1 \leq  |n| \leq  Z} \legendre{n}{p} \frac{1-\cos(2\pi n t)}{n} + O\left(\frac{1}{\sqrt{p}}+\frac{\sqrt{p}\log p}{Z}\right),
    \end{equation}
    if $p\equiv 3\pmod 4$. As before, we can view each $p \mapsto f_p$ for $Q\leq p\leq 2Q$ and $p\equiv 1\pmod 4$ (respectively $p\equiv 3\pmod 4$) as a random variable on the finite probability space $\{Q\leq p\leq 2Q \text{ and } p\equiv 1\pmod 4\}$ (respectively $\{Q\leq p\leq 2Q \text{ and } p\equiv 3\pmod 4\}$) endowed with the uniform probability measure, and we denote this random variable by  $\mathcal{F}_{Q, +}$ (respectively $\mathcal{F}_{Q, -}$).  We also 
do the same with the random Fourier series, where we fix $\X_{-1} = 1$ or $-1$. To this end we define for $t\in [0, 1]$
        \begin{align*}
            &F_{\X,+}(t) = \frac{1}{\pi} \sum_{n \geq 1} \X_n \frac{\sin(2\pi nt)}{n}, &\quad &F_{\X,-}(t) = \frac{1}{\pi} \sum_{n \geq 1} \X_n \frac{1 - \cos(2 \pi n t)}{n}.
        \end{align*}
       Then, it follows from the proof of Theorem \ref{thm: main theorem} that the sequence of processes $(\mathcal{F}_{Q, \pm})_{Q}$ converges to the process $F_{\X, \pm}$.
        \end{rem} 

    Recently, the second author \cite{La} investigated the distribution of $\frac{1}{\sqrt{p}}\max_{1\leq x\leq p} |S_p(x)|$, as $p$ varies over the primes $Q\leq p\leq 2Q$. In particular, he proved that the tail of the distribution is double exponentially decreasing in a large uniform range, which is believed to be best possible.  Since the maximum of the Legendre path $f_p$ occurs at its vertices and $||\cdot||_{\infty}$ is a continuous function on $C([0,1])$, a direct consequence of Theorem \ref{thm: main theorem} is the existence of a limiting distribution of $\frac{1}{\sqrt{p}}\max_{1\leq x\leq p} |S_p(x)|$.

    \begin{cor}\label{Cor: Distribution Maximum} Let $\mu$ be the probability measure on $[0, \infty)$ associated to the random variable $||F_{\X}||_{\infty}$. Then for any bounded
continuous function $h$ on $[0, \infty)$, we have 
$$ \lim_{Q \to \infty} \frac{1}{\pi^*(Q)}\sum_{Q\leq p\leq 2Q} h\left(\frac{1}{\sqrt{p}}\max_{1\leq x\leq p} |S_p(x)|\right)= \int_0^{\infty} h(x) d\mu (x).$$
    \end{cor}   
    This result is the analogue of Theorem 1.4 of \cite{BGGK}, which was proved (using a different and more direct approach) by Bober, Goldmakher, Granville and Koukoulopoulos for the family of non-principal characters modulo a large prime $q$.
    \subsection*{Acknowledgments} The authors would like to thank the anonymous referees for carefully reading the paper, and for several helpful suggestions, which led in particular to the simplification of the proof of Proposition \ref{ThirdRangeTSProposition}, as well as to the new Section 3.2 on the convergence of the sequence of Fourier coefficients of the path in  the sense of finite distributions. The second author is supported by a junior chair of the Institut Universitaire de France. Part of this work was completed while the second author was on a D\'el\'egation CNRS  at the IRL3457 CRM-CNRS in Montr\'eal. The second author would like to thank the CNRS for its support and the Centre de Recherches Math\'ematiques for its excellent working conditions. 
   
\section{Properties of the random Fourier series $F_{\X}(t)$}\label{s: Random Fourier}

\subsection{Almost sure continuity of $F_{\X}$}

Recall that $F_{\X}(t)$ is defined as the limit of the finite sum $F_{\X, N}(t)$ as $N\to \infty$, and that for fixed $t\in [0, 1]$,  the series defining $F_{\X}(t)$ is almost surely convergent by Lemma 1 of \cite{Ka}. In order to prove that $F_{\X}$ is almost surely continuous, we use the fact that $F_{\X, N}$ are $C([0,1])$-valued random variables, and then prove that the sequence of processes $(F_{\X, N})_{N}$ converges in law to the process $F_{\X}$. To this end we will show that  $(F_{\X, N})_{N}$ converges to $F_{\X}$ in the sense of finite distributions, and then show that the sequence $(F_{\X, N})_{N}$ is tight. 

\begin{pro}\label{pro: ContinuityRandom }
As $N\to \infty$, the sequence of processes $(F_{\X, N})_{N}$ converges in the Banach space $C([0, 1])$. In particular, its limit $F_{\X}$ is almost surely the Fourier series of a continuous function. 
\end{pro}

\begin{proof}
Let $t\in [0, 1]$. Then, we have 
\begin{align*}\ex\left(|F_{\X}(t)- F_{\X, N}(t)|^2\right)
&= \frac{1}{(2\pi)^2}\ex\Big(\Big|\sum_{|n|>N} \X_n \frac{1-e(-nt)}{n}\Big|^2\Big)\\
& = \frac{1}{(2\pi)^2}\sum_{|n_1|, |n_2|>N} \frac{(1-e(-n_1t))(1-e(n_2t))}{n_1n_2}\ex\left(\X_{n_1n_2}\right).
\end{align*}
Since $\ex(\X_n)=1$ if $n$ is a square, and equals $0$ otherwise, we deduce that 
$$ \ex\left(|F_{\X}(t)- F_{\X, N}(t)|^2\right) \ll \sum_{\substack{|n_1|, |n_2|>N \\ n_1n_2=\square}} \frac{1}{|n_1n_2|}\ll \sum_{n> N} \frac{d(n^2)}{n^2} \ll \sum_{n> N} \frac{1}{n^{3/2}} \ll \frac{1}{N^{1/2}},$$
upon using the bound $d(m)\ll m^{1/4}$, where $d(m)$ is the divisor function. Therefore, $F_{\X, N}$ converges to $F_{\X}$ in $L^{2}([0, 1])$ and hence in $L^1([0, 1])$ as $N\to \infty$. By Lemma B.11.3 of \cite{Ko} we deduce that the sequence of processes $(F_{\X, N})_{N}$ converges to the process $F_{\X}$ in the sense of finite distributions. 

To complete the proof, we need to show that the sequence of processes $(F_{\X, N})_{N\geq 1}$ is tight. By Kolmogorov's tightness criterion (see for example Proposition B.11.10 of \cite{Ko}), we need to prove the existence of constants $C\geq 0$, $\alpha>0$ and $\delta>0$ such that for any $N\geq 1$, and real numbers $0\leq s<t\leq 1$ we have 
$$\ex\left(\left|F_{\X, N}(t)-F_{\X, N}(s)\right|^{\alpha}\right)\leq C |t-s|^{1+\delta}.$$
We will prove this with $\alpha=4$ and $\delta= 1/2.$ Let $0\leq s<t\leq 1$ be real numbers.  For an integer $n\neq 0$ we define 
\begin{equation}\label{Definitiong}
g(n):= \sum_{\substack{n_1, n_2 \in \mathbb{Z}\\ n_1n_2=n}} \big(e(-n_1s)-e(-n_1t)\big)\big(e(-n_2s)-e(-n_2t)\big).
\end{equation}
Then we have 
\begin{equation}\label{eq: Bound4thMomentPartialRandom}
\begin{aligned}
\ex\left(\left|F_{\X, N}(t)-F_{\X, N}(s)\right|^{4}\right)&\ll \ex\Big(\Big|\sum_{1\leq |n|\leq N} \X_n \frac{e(-ns)-e(-nt)}{n}\Big|^4\Big)\\
& = \ex\Big(\Big|\sum_{1\leq |n|\leq N^2}  \frac{g(n) \X_n }{n}\Big|^2\Big) = \sum_{\substack{1\leq |n_1|, |n_2|\leq N^2\\ n_1n_2=\square}} \frac{g(n_1)\overline{g(n_2)}}{n_1n_2}. 
\end{aligned}
\end{equation}
Let $\ep>0$ be a small real number. Note that $|e(\alpha_1 n)-e(\alpha_2 n)|\ll \min(1, |n(\alpha_1-\alpha_2)|)$, and hence
\begin{equation}\label{EstimateGn}
\begin{aligned}    
g(n)&\ll \sum_{\substack{n_1, n_2 \in \mathbb{Z}\\ n_1n_2=n}} \min(1, |n_1(t-s)|)\min(1, |n_2(t-s)|)\ll d(|n|)\min(1, |n|(t-s)^2)\\
&\ll_{\ep} |n|^{\ep/4}\min(1, |n|(t-s)^2), 
\end{aligned}
\end{equation}
since $d(|n|)\ll_{\ep} |n|^{\ep/4}.$
This implies
$$
\sum_{\substack{1\leq |n_1|, |n_2|\leq N^2\\ n_1n_2=\square}} \left|\frac{g(n_1)g(n_2)}{n_1n_2}\right|
 \ll_{\ep} \sum_{\substack{1\leq |n_1|, |n_2|\leq N^2\\ n_1n_2=\square}} \frac{\min(1, |n_1n_2|(t-s)^4)}{|n_1n_2|^{1-\ep/4}}\ll_{\ep} \sum_{n=1}^{\infty}  d(n^2)\frac{\min(1, n^2(t-s)^4)}{n^{2-\ep/2}}. 
$$
We now use the bound $d(n^2)\leq d(n)^2\ll_{\ep} n^{\ep/2}$ to get 
\begin{equation}\label{BoundsGSum}
\sum_{\substack{1\leq |n_1|, |n_2|\leq N^2\\ n_1n_2=\square}} \left|\frac{g(n_1)g(n_2)}{n_1n_2}\right|
\ll_{\ep} 
\sum_{n \geq \frac{1}{|t-s|^2}} \frac{1}{n^{2-\ep}} + |t-s|^4\sum_{n \leq \frac{1}{|t-s|^2}} n^{\ep}
\ll_{\ep} |t-s|^{2-2\ep}.
\end{equation}
Inserting this estimate in \eqref{eq: Bound4thMomentPartialRandom} and choosing $\ep=1/4$ implies that 
$$\ex\left(\left|F_{\X, N}(t)-F_{\X, N}(s)\right|^{4}\right)\leq C |t-s|^{3/2},$$
for all positive integers $N$ and real numbers $0\leq s<t\leq 1$, where $C>0$ is an absolute constant. This completes the proof.  
\end{proof}

\subsection{The sequence of joint moments of the random process is determinate.}\label{s: the limiting moment}
In Section \ref{s: proof of convergence in finite distributions real} we shall use the method of moments to prove that the sequence of processes $(\mathcal{F}_{Q})_Q$ converges to the process $F_{\X}$ in the sense of convergence of finite distributions, as $Q\to \infty$.  More precisely, given an integer $k \geq 1$ and a $k-$tuple $0 \leq t_1 < \cdots < t_k \leq 1$, we will prove that the joint moments of $(\mathcal{F}_{Q}(t_1), \dots, \mathcal{F}_{Q}(t_k)) $  
    converge to those of $(F_{\X}(t_1), \dots, F_{\X}(t_k))$ as $Q\to \infty$. However, we first need to verify that these moments are determinate, i.e. that they have only one representing measure. For $\mathbf{n}=(n_1, \dots, n_k)\in (\mathbb{Z}_{\geq 0})^k$ we let $n= \sum_{j=1}^k n_j$, and define 
    \begin{equation}\label{eq: limiting moment}
            M_{\X}(\mathbf{n}):= \mathbb{E}\left( \prod_{i = 1}^k F_{\X}(t_i)^{n_i} \right), \quad 
            M_{\X,\pm}(\mathbf{n}):= \mathbb{E}\left( \prod_{i = 1}^k F_{\X,\pm}(t_i)^{n_i} \right).
        \end{equation}
Then we have 
\begin{equation}\label{eq: MomentRandomPM}
\begin{aligned}
 M_{\X}(\mathbf{n})&= \mathbb{E}\Bigg( \prod_{i = 1}^k F_{\X}(t_i)^{n_i} \ \Big|  \ \X_{-1}=1\Bigg) \pr(\X_{-1}=1) \\ 
 & \quad \quad  \quad +\mathbb{E}\left( \prod_{i = 1}^k F_{\X}(t_i)^{n_i} \ \Big| \ \X_{-1}=-1\right) \pr(\X_{-1}=-1)\\
& =\frac{1}{2}\left(M_{\X,+}(\mathbf{n})+ M_{\X,-}(\mathbf{n})\right), 
\end{aligned}
\end{equation}
since $\pr(\X_{-1}=1)= \pr(\X_{-1}=-1)=1/2.$
Therefore, using the definition of $F_{\X,+}(t)$ and expanding the moment $M_{\X,+}(\mathbf{n})$ we obtain
\begin{equation}\label{eq: FormulaMomentRandom+}
    \begin{aligned}
        M_{\X,+}(\mathbf{n}) &= \frac{1}{\pi^{n}} \mathbb{E} \left( \sum_{a_{1,1}, \dots, a_{k,n_k} \geq 1} \prod_{i = 1}^k \prod_{j = 1}^{n_i} \X_{a_{i,j}} \frac{\sin(2 \pi a_{i,j}t_i)}{a_{i,j}}\right)\\ & = \frac{1}{\pi^{n}} \mathbb{E} \left(\sum_{a=1}^{\infty} \frac{\X_a}{a}\mathcal{B}_{\mathbf{n},\mathbf{t}, +}(a)\right)= \frac{1}{\pi^{n}} \sum_{a=1}^{\infty} \frac{\mathcal{B}_{\mathbf{n},\mathbf{t}, +}(a^2)}{a^2},
    \end{aligned}
    \end{equation}
     where
    \begin{align*}
        \mathcal{B}_{\mathbf{n},\mathbf{t},+}(a) 
        := \sum_{a_{1}\cdots a_{k} = a} \prod_{i = 1}^k \sum_{b_{i,1} \cdots b_{i,n_i} = a_i } \prod_{j = 1}^{n_i} \sin(2 \pi b_{i,j} t_i).
    \end{align*}
    We observe that uniformly for $\mathbf{t}$ we have 
    \begin{equation}\label{eq: BoundBnt+}
|\mathcal{B}_{\mathbf{n},\mathbf{t}, +}(a)| \leq \sum_{a_1 \cdots a_k = a} \prod_{i = 1}^k d_{n_i}(a_i) = d_n (a),
\end{equation}
    where $d_n$ is the $n$-th divisor function.
Similarly, one has 
\begin{equation}\label{eq: FormulaMomentRandom-}
M_{\X,-}(\mathbf{n})= \frac{1}{\pi^{n}} \sum_{a=1}^{\infty} \frac{\mathcal{B}_{\mathbf{n},\mathbf{t}, -}(a^2)}{a^2},
\end{equation}    
where \begin{align*}
\mathcal{B}_{\mathbf{n},\mathbf{t},-}(a) 
:= \sum_{a_{1}\cdots a_{k} = a} \prod_{i = 1}^k \sum_{b_{i,1}\cdots b_{i,n_i} = a_i } \prod_{j = 1}^{n_i} (1-\cos(2 \pi b_{i,j} t_i)).
\end{align*}
We also have the following analogous bound to \eqref{eq: BoundBnt+}
\begin{equation}\label{eq: BoundBnt-}|\mathcal{B}_{\mathbf{n},\mathbf{t}, -}(a)|\leq 2^nd_n(a).
\end{equation}
    
   We now prove the following lemma.
    \begin{lem}
        The sequence of moments $(M_{\X}(\mathbf{n}))_{\mathbf{n}\in (\mathbb{Z}_{\geq 0})^k}$ have only one representing measure. 
    \end{lem}

    \begin{proof}
   It is sufficient to  show that the moments satisfy the  
    Carleman condition \cite[Theorem 15.11]{Sc}:
    \begin{align}\label{eq: carleman condition}
            \sum_{n = 1}^\infty \left|\ex\big(F_{\X}(t_i)^{2n}\big)\right|^{-\frac{1}{2n}} = \infty,
         \end{align}
        for all $1\leq i\leq k$. 
            
        Let $t\in [0, 1]$ be fixed. Combining the identities \eqref{eq: MomentRandomPM}, \eqref{eq: FormulaMomentRandom+}, and \eqref{eq: FormulaMomentRandom-} with the upper bounds \eqref{eq: BoundBnt+} and \eqref{eq: BoundBnt-} (with $k=1$) we deduce that 
       $$
        \left|\ex\big(F_{\X}(t)^{2n}\big)\right| 
              \leq \frac{2^{2n}+1}{2\pi^{2n}} \sum_{a=1}^{\infty}\frac{d_{2n}(a^2)}{a^2}.
      $$  
        Now, it follows from Lemma 3.3 of \cite{La2} that 
        $$\sum_{a=1}^{\infty}\frac{d_{2n}(a^2)}{a^2} \leq \sum_{a=1}^{\infty}\frac{d_{2n}(a)^2}{a^2}= \exp\big(O(n\log\log n)\big).$$
        Combining these estimates implies the existence of  positive constants $C_1, C_2$ such that for all integers $n\geq 2$ 
        $$ \left|\ex\big(F_{\X}(t)^{2n}\big)\right|^{-\frac{1}{2n}} \geq C_2 \frac{1}{(\log n)^{C_1}},$$
        and hence 
        $$ \sum_{n=2}^{\infty}\left|\ex\big(F_{\X}(t)^{2n}\big)\right|^{-\frac{1}{2n}}\geq C_2 \sum_{n=2}^{\infty}\frac{1}{(\log n)^{C_1}}= \infty.$$
        This concludes the proof.
        
    \end{proof}

\section{Proof of convergence in the sense of finite distributions}\label{s: proof of convergence in finite distributions real}

\subsection{Convergence in the sense of finite distributions of the process}     In this section we will prove the following theorem.
        \begin{thm}\label{thm: convergence in finite distribution}
            As $Q\to \infty$ the sequence of processes $(\mathcal{F}_{Q})_Q$ converges to the process $F_{\X}$ in the sense of convergence of finite distributions.
            In other words, for every $k \geq 1$ and for every $k-$tuple $0 \leq t_1 < \cdots < t_k \leq 1$, the vectors 
           $ (\mathcal{F}_{Q}(t_1), \dots, \mathcal{F}_{Q}(t_k)) 
            $  
            converge in law, as $Q \rightarrow \infty$, to
                $(F_{\X}(t_1), \dots, F_{\X}(t_n))$. 
          
        \end{thm}

        
        \noindent In order to prove this result it suffices to show that the joint moments of $(\mathcal{F}_{Q}(t_1), \dots, \mathcal{F}_{Q}(t_n))$ converge to those of $(F_{\X}(t_1), \dots, F_{\X}(t_n))$ as $Q\to \infty$, since it follows from Section \ref{s: the limiting moment} that the distribution of the latter random vector is completely determined by its joint moments. Therefore, Theorem \ref{thm: convergence in finite distribution} is a direct consequence of the following proposition. 
      
        \begin{pro}\label{thm: moments theorem}
            Let $A>1$ be a fixed constant. Let $k$ be a positive integer, $\mathbf{n} = (n_1, \dots, n_k)\in (\mathbb{Z}_{\geq 0})^k$ and $0 \leq t_1 < \cdots < t_k \leq 1$ be real numbers. Then we have
\begin{equation}\label{AsymptoticMomentsMQ}
                \frac{1}{\pi^*(Q)} \sum_{Q\leq p \leq 2Q} \prod_{i =1}^k f_{p}(t_i)^{n_i}= M_{\X}(\mathbf{n}) + O_{\mathbf{n}, A}\left(\frac{1}{(\log Q)^{A}}\right), 
            \end{equation}
          where $M_{\X}(\mathbf{n})$ is defined in \eqref{eq: limiting moment}, and the implicit constant in the error term is not effective. 
        \end{pro}
        
Conditionally on GRH, the first author established this result in her thesis \cite{Hu}, with a much stronger error term $Q^{-1/2+o(1)}$. In order to compute these moments unconditionally, the main ingredient is a recent work of the second author \cite{La}, which allows us to substantially reduce the length of the P\'olya Fourier expansion of $f_p$, for almost all primes $Q\leq p\leq 2Q$ (see Lemma \ref{lem: Tail2} below).
Another difficulty in proving Proposition \ref{thm: moments theorem} (which explains the weak error term of \eqref{AsymptoticMomentsMQ}) arises from the possible existence of Landau-Siegel zeros.  Let $x$ be a large real number. By the results of Chapter 20 of \cite{Da}, it follows that for all squarefree integers $|d|\leq \exp(\sqrt{\log x})$ with at most one exception $q_1$, we have
\begin{equation}\label{NoExceptional}
\sum_{p\leq x} \chi_d(p) \ll x\exp\left(-c\sqrt{\log x}\right),
\end{equation}
for some positive constant $c$, where $\chi_d$ denotes the Kronecker symbol attached to $d$.
Moreover, by Siegel's Theorem 
this exceptional discriminant if it exists must satisfy
\begin{equation}\label{LowerBoundSiegel}
 |q_1|\gg_A (\log x)^{4A},
 \end{equation}
 for any fixed constant $A>0$, where the implicit constant is not effective.
 To lighten our notation we shall denote by $\chi_p:=\left(\frac{\cdot}{p}\right)$ the Legendre symbol modulo a prime $p$, throughout the remaining part of the paper. 
 The following lemma is standard (see  \cite{Da}) but we include its proof for the sake of completeness. 
  \begin{lem}\label{Siegel}
Let $Q$ be large and $n\leq \exp\left(\sqrt{\log Q}\right)$ be a positive integer. Then, we have
$$
\sum_{\substack{Q\leq p\leq 2Q\\ p\equiv 1 \bmod 4}} \chi_{p}(n)=\begin{cases} \displaystyle{\frac{\pi^*(Q)}{2}}+O\left(Q\exp\left(-c\sqrt{\log Q}\right)\right), & \textup{ if } n= m^2,\\ 
O\left(Q\exp\left(-c\sqrt{\log Q}\right)\right), & \textup{ if } n \text{ is not of the form } m^2 \text{ or } |q_1| \cdot m^2,\end{cases}
$$
for some positive constant $c$.
Moreover, the same estimate holds for $ \sum_{\substack{Q\leq p\leq 2Q\\ p\equiv 3 \bmod 4}} \chi_{p}(n)$. 
\end{lem}
\begin{proof} We only prove the estimate for the character sum over primes $p\equiv 1 \pmod 4$ since the proof in the case $p\equiv 3 \pmod 4$ is similar.
 The first estimate when $n$ is a square follows simply from the prime number theorem in arithmetic progressions together with the trivial bound $\omega(n)\leq \log n$, where $\omega(n)$ is the number of prime factors of $n$. Now, suppose that $n$ is not a square, and write $n=dm^2$ where $d$ is squarefree and $d\neq |q_1|$. By the law of quadratic reciprocity we deduce that 
$$ \sum_{\substack{Q\leq p\leq 2Q\\ p\equiv 1 \bmod 4}} \chi_p(n)= \sum_{\substack{Q\leq p\leq 2Q\\ p\equiv 1 \bmod 4}}\chi_p(d) + O(\omega(n))=  \sum_{\substack{Q\leq p\leq 2Q\\ p\equiv 1 \bmod 4}}\chi_d(p) + O(\sqrt{\log Q}).
$$ 
The second estimate follows from \eqref{NoExceptional}, together with the fact that 
$$\sum_{\substack{Q\leq p\leq 2Q\\ p\equiv 1 \bmod 4}}\chi_d(p)= \frac{1}{2} \sum_{Q\leq p\leq 2Q}\chi_d(p) + \frac{1}{2}\sum_{p\leq Q}\chi_{-d}(p).$$

\end{proof}

In order to compute the joint moments of $f_{p}(t_1), \dots, f_{p}(t_k)$ we need to considerably shorten the sum in the main term of \eqref{eq: definition Legendre path}, for all $t\in \{t_1, t_2, \dots, t_k\}$. To this end we shall use the following result from a recent work of the second author \cite{La} on the distribution of large values of quadratic character sums. 

\begin{lem}\label{lem: Tail2} Let $B>0$ be a fixed constant. For all real numbers $N_1, N_2$ verifying  

\noindent $ \exp\big((\log\log Q)^2\big)\leq N_1 < N_2\leq 2N_1\leq Q^{21/40}$, the number of primes $Q\leq p\leq 2Q$ such that 
$$ 
\max_{\alpha\in [0, 1)}\left|\sum_{N_1\leq n\leq N_2} \frac{\chi_p(n) e(\alpha n)}{n}\right|\geq \frac{1}{(\log N_1)^B},
$$
is 
$$ \ll_B Q\exp\left(-\frac{\log Q}{10 \log\log Q}\right).$$
\end{lem}
\begin{proof} This is a direct consequence of Propositions 4.2 and 4.3 of \cite{La} by embedding the set of primes $Q\leq p\leq 2Q$ into the set of all fundamental discriminants $|d|\leq 2Q$.
\end{proof}

We now have all the ingredients to prove Proposition \ref{thm: moments theorem}.

\begin{proof}[Proof of Proposition \ref{thm: moments theorem}] 
Let $k$ be a positive integer, $\mathbf{n} = (n_1, \dots, n_k)\in (\mathbb{Z}_{\geq 0})^k$ and $0 \leq t_1 < \cdots < t_k \leq 1$ be real numbers. 
Throughout the proof we shall assume that $Q$ is sufficiently large in terms of $n=n_1+\cdots+n_k$. To shorten our notation we define 
             $$ 
            M_{Q}(\mathbf{n}):= \frac{1}{\pi^*(Q)} \sum_{Q\leq p \leq 2Q} \prod_{i =1}^k f_{p}(t_i)^{n_i}.
        $$
Then, we note that 
$$M_{Q}(\mathbf{n})= \frac{1}{2}\left(M_{Q,+}(\mathbf{n})+ M_{Q,-}(\mathbf{n})\right),$$
where
        \begin{equation}\label{eq: moment sequence}
            M_{Q,+}(\mathbf{n}):= \frac{2}{\pi^*(Q)} \sum_{\substack{Q\leq p \leq 2Q\\ p \equiv 1 \, \bmod 4}} \prod_{i =1}^k f_{p}(t_i)^{n_i}, \quad M_{Q,-}(\mathbf{n}):= \frac{2}{\pi^*(Q)} \sum_{\substack{Q\leq p \leq 2Q\\ p \equiv 3 \, \bmod 4}} \prod_{i =1}^k f_{p}(t_i)^{n_i}.
        \end{equation}
       Therefore, in view of \eqref{eq: MomentRandomPM} it suffices to establish the following asymptotic formulas 
        \begin{equation}\label{eq: AsymptoticPlusMinus}
             M_{Q, +}(\mathbf{n}) = M_{\X, +}(\mathbf{n}) + O_{\mathbf{n},A}\left(\frac{1}{(\log Q)^{A}}\right), \quad 
             M_{Q, -}(\mathbf{n}) = M_{\X, -}(\mathbf{n}) + O_{\mathbf{n},A}\left(\frac{1}{(\log Q)^{A}}\right).
        \end{equation}
       We shall only prove this estimate for $M_{Q, +}(\mathbf{n})$, since the proof for $M_{Q, -}(\un)$ is similar. 
       
       Using \eqref{eq: definition Legendre path} with $Z:= Q^{21/40}$ we get

$$
M_{Q, +}(\un)= \frac{1}{\pi^n} \frac{2}{\pi^*(Q)} \sum_{\substack{Q\leq p\leq 2Q\\ p\equiv 1\bmod 4}}\prod_{i=1}^k \left(\sum_{a\leq Z} \chi_p(a) \frac{\sin(2\pi at_i)}{a}+ O\left(Q^{-1/50}\right)\right)^{n_i}.
$$
We now define $Y:=\exp((\log Q)^{1/3})$. 
Let $L_1:=\lfloor \log Y /\log 2\rfloor$, $L_2:= \lfloor \log Z/\log 2\rfloor$, and put $s_{L_1}:= Y$, $s_{L_2+1}:= Z$, and $s_{\ell}:=2^{\ell}$ for $L_1+1\leq \ell\leq L_2$.
Then we have
  \begin{equation}\label{PolyaSplit}
\begin{aligned}
\max_{t\in [0, 1)} \left|\sum_{Y\leq a\leq Z} \chi_p(a) \frac{\sin(2\pi at)}{a}\right| 
&  \leq   \sum_{L_1\leq \ell \leq L_2} \max_{t \in [0, 1)} \Bigg|\sum_{s_{\ell}\leq a\leq s_{\ell+1}}  \frac{\chi_p(a)e(at)}{a}\Bigg|.
\end{aligned}
\end{equation}
Furthermore, it follows from Lemma \ref{lem: Tail2} with the choice  $B=6An+1$ that 
\begin{equation}\label{BoundTailLarge}
\sum_{L_1\leq \ell \leq L_2} \max_{t \in [0, 1)} \Bigg|\sum_{s_{\ell}\leq a\leq s_{\ell+1}}  \frac{\chi_p(a)e(at)}{a}\Bigg|
\ll  
\sum_{L_1\leq \ell \leq L_2} \frac{1}{\ell^B} \ll \frac{1}{(\log Q)^{2An}},
\end{equation}
for all primes $Q\leq p\leq 2Q$ except for a set $\mathcal{E}(Q)$ of size 
$$ |\mathcal{E}(Q)| \ll_n L_2 Q \exp\left(-\frac{\log Q }{10 \log\log Q}\right)  \ll_n Q \exp\left(-\frac{\log Q}{20 \log\log Q}\right).$$
Combining the estimates \eqref{PolyaSplit} and \eqref{BoundTailLarge} we deduce that $M_{Q, +}(\un)$ equals
\begin{align*}
&\frac{2}{\pi^n\pi^*(Q)} \sum_{\substack{p\in [Q, 2Q]\setminus \mathcal{E}(Q)\\ p\equiv 1\bmod 4}}\prod_{i=1}^k \left(\sum_{a\leq Z} \chi_p(a) \frac{\sin(2\pi at_i)}{a}+ O\left(Q^{-1/50}\right)\right)^{n_i}
+ O_n\left(\exp\left(-\frac{\log Q}{30\log\log Q}\right)\right)\\
&= \frac{2}{\pi^n\pi^*(Q)}  \sum_{\substack{p\in [Q, 2Q]\setminus \mathcal{E}(Q)\\ p\equiv 1\bmod 4}}\prod_{i=1}^k \left(\sum_{a\leq Y} \chi_p(a) \frac{\sin(2\pi at_i)}{a}+ O\left(\frac{1}{(\log Q)^{2An}}\right)\right)^{n_i}+ O_n\left(\exp\left(-\frac{\log Q}{30\log\log Q}\right)\right)\\
&= \frac{2}{\pi^n\pi^*(Q)}\sum_{\substack{p\leq Q\\ p\equiv 1\bmod 4}}\prod_{i=1}^k \left(\sum_{a\leq Y} \chi_p(a) \frac{\sin(2\pi at_i)}{a}+ O\left(\frac{1}{(\log Q)^{2An}}\right)\right)^{n_i}+ O_n\left(\exp\left(-\frac{\log Q}{30\log\log Q}\right)\right).\\
\end{align*}
Expanding the main term we obtain
\begin{equation}\label{MQMain}
\begin{aligned}
M_{Q, +}(\un)
&= \frac{2}{\pi^n\pi^*(Q)}\sum_{\substack{Q\leq p\leq 2Q\\ p\equiv 1\bmod 4}}\sum_{a\leq Y^n}\chi_p(a) \frac{\B_{Y,\un, \ut, +}(a)}{a} +O _n\left(\frac{(\log Y)^n}{(\log Q)^{2An}}\right)\\
&= \frac{2}{\pi^n}\sum_{a\leq Y^n}\frac{\B_{Y,\un, \ut, +}(a)}{a} \frac{1}{\pi^*(Q)}\sum_{\substack{Q\leq p\leq 2Q\\ p\equiv 1\bmod 4}}\chi_p(a)  +O_n\left(\frac{1}{(\log Q)^{An}}\right),
\end{aligned}
\end{equation}
 where $\B_{Y,\un, \ut, +}(a)$ is defined by
 $$\B_{Y,\un, \ut, +}(a)= 
        \sum_{a_{1}\cdots a_{k} = a} \prod_{i = 1}^k \sum_{\substack{b_{i,1}, \cdots b_{i,n_i} = a_i \\ b_{i,j} \leq  Y}} \prod_{j = 1}^{n_i} \sin(2 \pi b_{i,j} t_i).$$ 
 We shall first estimate the contribution of the squares $a=m^2$ which gives rise to the main term of $M_{Q, +}(\un)$. Indeed, the contribution of these terms to the right hand side of \eqref{MQMain} equals
$$ 
\frac{1}{\pi^n}\sum_{m\leq Y^{n/2}}\frac{\B_{Y,\un, \ut, +}(m^2)}{m^2} \left(1+O\left(\exp\left(-c_0\sqrt{\log Q}\right)\right)\right) + O\left(\frac{1}{\pi^*(Q)}\sum_{m\leq Y^{n/2}}\frac{|\B_{Y,\un, \ut, +}(m^2)|\omega(m)}{m^2}\right),
$$
for some positive constant $c_0$, by the Prime Number Theorem for arithmetic progressions.
 Let $\varepsilon>0$ be small. By \eqref{eq: BoundBnt+} we have $\B_{Y,\un, \ut, +}(m^2), \B_{\un, \ut, +}(m^2)\ll_{\varepsilon, n} m^{\varepsilon}$.  Since $\omega(m)\leq \log m$ and $\B_{Y,\un, \ut, +}(\ell)= \B_{\un, \ut, +}(\ell)$ if $\ell\leq Y$, we deduce that this contribution equals 
\begin{equation}\label{SquaresContrib}
\begin{aligned}
&\frac{1}{\pi^n}\sum_{m\leq \sqrt{Y}}\frac{\B_{Y,\un, \ut, +}(m^2)}{m^2} + O_n\left(\sum_{\sqrt{Y}<m\leq Y^{n/2}}\frac{1}{m^{2-\varepsilon}}+ \exp\left(-c_0\sqrt{\log Q}\right)\right)\\
= & \frac{1}{\pi^n}\sum_{m=1}^{\infty}\frac{\B_{\un, \ut, +}(m^2)}{m^2} + O_n\left(\sum_{m> \sqrt{Y}}\frac{1}{m^{2-\varepsilon}}+\exp\left(-c_0\sqrt{\log Q}\right)\right)\\
=  & M_{\X, +}(\un)+ O\left(Y^{-1/3}\right),\\
\end{aligned}
\end{equation}
by \eqref{eq: FormulaMomentRandom+}. 
Next, we bound the contribution of the terms $a$ which are not of the form $m^2$ or $|q_1|m^2$. Since $Y^n\leq \exp\left(\sqrt{\log Q}\right)$ by our assumptions on $Y$ and $Q$, it follows from Lemma \ref{Siegel} that the contribution of these terms to the main term on the right hand side of \eqref{MQMain} is 
\begin{equation}\label{NonSquares}
\ll \exp\left(-\frac{c}{2}\sqrt{\log Q}\right)\sum_{a\leq Y^n}\frac{|\B_{Y,\un, \ut, +}(a)|}{a} \ll Y^{n/2} \exp\left(-\frac{c}{2}\sqrt{\log Q}\right) \ll \exp\left(-\frac{c}{3}\sqrt{\log Q}\right),
\end{equation}
since $\B_{Y,\un, \ut, +}(a)\ll \sqrt{a}$, and where $c>0$ is the constant in Lemma \ref{Siegel}.

Finally, we bound the contribution of the terms $a=|q_1|m^2$. Using the bound $\B_{Y,\un, \ut, +}(a)\ll_{n} a^{1/4}$, and bounding the inner character sum in the right hand side of \eqref{MQMain} trivially, we deduce that the contribution of these terms is 
\begin{equation}\label{ExceptionalContrib}
\ll \sum_{m\leq Y^{n/2}/\sqrt{|q_1|}}\frac{\B_{Y,\un, \ut, +}(|q_1|m^2)}{|q_1|m^2}\ll_{n} |q_1|^{-3/4}\sum_{m\geq 1}\frac{1}{m^{3/2}}\ll_{A, n} (\log Q)^{-A},
\end{equation}
by \eqref{LowerBoundSiegel}.
Inserting the estimates \eqref{SquaresContrib}, \eqref{NonSquares}, and \eqref{ExceptionalContrib} in \eqref{MQMain} completes the proof.
\end{proof}

\subsection{Establishing convergence in the sense of finite distributions for the sequence of ``Fourier coefficients'' of the path}

In this section, we describe a different approach to establishing convergence in the sense of finite distributions. 
Let $C_0([0,1])$ be the subspace of $C([0,1])$ consisting of functions vanishing at
$0$, and let $C_0(\mathbb{Z})$ be the Banach space of complex-valued functions on
$\mathbb{Z}$ converging to $0$ at infinity with the sup norm. Then note that $f_p\in C_0([0,1])$ for all primes $p$. 
We will follow the notation in Chapter 6 of Kowalski's book \cite{Ko}.
Let $\textup{FT}_p= (\textup{FT}_p(h))_{h\in \mathbb{Z}}$ be defined by $\textup{FT}_p(0)=0$ and for $h\neq 0$
$$ \textup{FT}_p(h):= \int_0^1 f_p(t) e(-ht)dt. $$
By Proposition B.11.8 of \cite{Ko}, one can replace Theorem \ref{thm: convergence in finite distribution} (convergence in the sense of finite distributions of the process) in the proof of Theorem \ref{thm: main theorem}, by the convergence in the sense of finite distributions of the sequence of the Fourier coefficients $(\textup{FT}_p)_{Q\leq p\leq 2Q}$ in $C_0(\mathbb{Z})$. Here convergence in the sense of finite distributions of a sequence $(X_n)$ of $C_0(\mathbb{Z})$-valued random variables to $X$ means that for any $H\geq 1$, the vectors $(X_{n,h})_{|h|\leq H}$ converge in law to $(X_h)_{|h|\leq H}$, in the sense of convergence in law in $\mathbb{C}^{2H+1}$. 
More specifically, in our case one needs to compute the moments 
\begin{equation}\label{Eq.MomentsFourier}\frac{1}{\pi^*(Q)}\sum_{Q\leq p\leq 2Q} \prod_{0<|h|\leq H}\textup{FT}_p (h)^{j_h},
\end{equation}
where $\{j_h\}_{0<|h|\leq H}$ are non-negative integers, since $\textup{FT}_p(0)=0$ for all $p$. 

Let $p$ be a large prime number. It follows from formula (6.5) of  \cite[Chapter 6]{Ko} that 
$$ \textup{FT}_p(h)=\frac{1}{2\pi i h}\frac{\sin(\pi h/(p-1))}{\pi h/(p-1)} e\left(-\frac{-h}{2(p-1)}\right) \frac{1}{\sqrt{p}} \sum_{m=1}^{p-1} \legendre{m}{p} e\left(-\frac{hm}{p-1}\right), $$
for $h\neq 0$. Since 
$$ e\left(-\frac{hm}{p-1}\right)= e\left(-\frac{hm}{p}\right) +O_h\left(\frac{1}{p}\right),$$
for fixed $h$ and 
$$ \sum_{m=1}^{p-1} \legendre{m}{p} e\left(-\frac{hm}{p}\right)= \tau\left(\legendre{\cdot}{p}\right) \legendre{-h}{p},$$
we deduce that 
$$ \textup{FT}_p(h)= \frac{\varepsilon_p}{2\pi i h} \frac{\sin(\pi h/(p-1))}{\pi h/(p-1)} e\left(-\frac{-h}{2(p-1)}\right) \legendre{-h}{p} +O_h\left(\frac{1}{\sqrt{p}}\right), $$
where $\ep_p$ is defined in \eqref{eq: PolyaFourier}. By using the Taylor expansions of $\sin(\pi h/(p-1))$ and $e(-h/(2(p-1)))$ we deduce that 
$$\textup{FT}_p(h)= \frac{\ep_p}{2\pi i h}\legendre{-h}{p} +O_h\left(\frac{1}{\sqrt{p}}\right). $$
Using this estimate and taking $H$ and $\mathbf{j}=(j_h)_{0<|h|\leq H}$ fixed we obtain \begin{align*}
    \frac{1}{\pi^*(Q)}\sum_{Q\leq p\leq 2Q} \prod_{0<|h|\leq H}\textup{FT}_p (h)^{j_h}
    &= \prod_{0<|h|\leq H}\left(\frac{1}{2\pi i h}\right)^{j_h}\frac{1}{\pi^*(Q)}\sum_{\substack{Q\leq p\leq 2Q\\ p\equiv 1 \bmod 4}} \legendre{\prod_{0<|h|\leq H} (-h)^{j_h}}{p}\\
    &+\prod_{0<|h|\leq H}\left(\frac{1}{2\pi  h}\right)^{j_h}\frac{1}{\pi^*(Q)}\sum_{\substack{Q\leq p\leq 2Q\\ p\equiv 3 \bmod 4}} \legendre{\prod_{0<|h|\leq H} (-h)^{j_h}}{p} \\
    & \quad +O_{H, \mathbf{j}}\left(\frac{1}{\sqrt{Q}}\right).
\end{align*} 
Therefore, using Lemma \ref{Siegel} and the bound \eqref{LowerBoundSiegel} together with the fact that $H$ and $\mathbf{j}$ are fixed, we get
$$ 
\frac{1}{\pi^*(Q)}\sum_{Q\leq p\leq 2Q} \prod_{0<|h|\leq H}\textup{FT}_p (h)^{j_h}= \ex\left(\prod_{0<|h|\leq H}\left(\frac{\Y}{2\pi i h} \X_{-h} \right)^{j_h}\right) +O_{H, \mathbf{j}}\left(\exp\left(-\frac c2 \sqrt{\log Q}\right)\right),
$$
where $c$ is the constant from Lemma \ref{Siegel}.  This shows that the Fourier coefficients of the Legendre path converge to those of $F_{\mathbb{X}}$ in the sense of finite distributions, as desired.

\section{Tightness of the Sequence of processes $\mathcal{F}_Q$}\label{s: Tighness}

    To complete the proof of Theorem \ref{thm: main theorem} we need to show that the sequence of processes $(\mathcal{F}_{Q})_Q$ is tight. To this end we shall use the following variant of Kolmogorov's tightness criterion. 

\begin{pro}[Proposition B.11.11 of \cite{Ko}] Let $(\Y_n)_{n\geq 1}$ be a sequence of $C([0,1])$-valued random variables. Suppose that there exist positive real numbers $\alpha_1, \alpha_2, \alpha_3, \beta_1,\beta_2, \delta, C$, such that $\beta_2<\beta_1$ and
for any real numbers $0\leq s<t\leq 1$ and any $n\geq 1$, we have 
$$
\begin{cases}
\ex(|\Y_n(t)-\Y_n(s)|^{\alpha_1})\leq C|t-s|^{1+\delta} & \text{ if } 0< |t-s|\leq n^{-\beta_1}, \\
\ex(|\Y_n(t)-\Y_n(s)|^{\alpha_2})\leq C|t-s|^{1+\delta} & \text{ if } n^{-\beta_1}\leq |t-s|\leq n^{-\beta_2}, \\
\ex(|\Y_n(t)-\Y_n(s)|^{\alpha_3})\leq C|t-s|^{1+\delta} & \text{ if } n^{-\beta_2}\leq |t-s|\leq 1. 
\end{cases}
$$
Then $(\Y_n)_{n\geq 1}$ is tight.

\end{pro}

Let $0\leq s<t\leq 1$ be real numbers. To use this criterion for the sequence $(\mathcal{F}_{Q})_Q$, we shall bound different moments of $|f_p(t)-f_p(s)|$ (as $p$ varies in $[Q, 2Q]$) in the three ranges $0< |t-s|\leq Q^{-1}$,  $Q^{-1}\leq |t-s|\leq Q^{-2/5}$, and $ Q^{-2/5}\leq |t-s|\leq 1.$ 

\subsection{The ``trivial'' range: $0< |t-s|\leq 1/Q$}
Recall the definition of $f_p$: the concatenation of points in the time graph of $\frac{1}{\sqrt{p}} \sum_{n \leq pt} \chi_p(n)$. This implies the trivial bound
$$
|f_p(t) - f_p(s)| \leq \frac{1}{\sqrt{p}}|pt - ps| = \sqrt{p}|t-s|.
$$
Therefore, if $0<|t-s|\leq 1/Q$ then $|f_p(t) - f_p(s)|\leq |t-s|^{1/2}$ and hence
\begin{equation}\label{FirstRangeTS}  
\frac{1}{\pi^*(Q)}\sum_{Q\leq p\leq 2Q} 
|f_p(t)-f_p(s)|^4 \leq |t-s|^2.
\end{equation}

\begin{rem}\label{Rem.Trivial}
By taking a very large moment of $|f_p(t)-f_p(s)|$ we can cover the whole range $0 \leq |t-s|\leq Q^{-1/2-\ep}$.
\end{rem}


\subsection{The ``Burgess'' range: $Q^{-1} \leq |t-s|\leq Q^{-2/5}$} In this range we shall bound a very large moment of $|f_p(t)-f_p(s)|$ using the following bound of Burgess for short character sums  (see \cite[Theorem 1]{Bu} with $r=2$) 
\begin{equation}\label{Eq.Burgess}
    \sum_{M\leq n\leq M+N} \chi_p(n) \ll N^{1/2} p^{3/16}\log p,
\end{equation} 
which is valid uniformly in $N$. 
\begin{lem}\label{SecondRangeTSLemma}
Let $0\leq s<t\leq 1$ be real numbers such that $Q^{-1}\leq |t-s|\leq Q^{-2/5}$. Then we have
$$
 \frac{1}{\pi^*(Q)}\sum_{Q\leq p\leq 2Q} |f_p(t)-f_p(s)|^{1000} \ll |t-s|^2.
$$

\end{lem}

\begin{proof}
It follows from \eqref{Eq.Burgess} that
$$ |f_p(t) - f_p(s)| \leq \frac{1}{\sqrt{p}}\left|\sum_{ps\leq n\leq pt} \chi_p(n)\right| +O\left(\frac{1}{\sqrt{p}}\right) \ll p^{3/16} (\log p) |t-s|^{1/2},$$
since $ |t-s|\geq 1/Q$. 
Therefore, using that $|t-s|\leq Q^{-2/5}$ we obtain
$$ \frac{1}{\pi^*(Q)}\sum_{Q\leq p\leq 2Q} |f_p(t)-f_p(s)|^{1000} \ll Q^{375/2} (\log Q)^{1000} |t-s|^{500}  \ll |t-s|^2,$$
as desired. 
\end{proof}
\begin{rem}\label{Rem.Burgess}
By taking a very large moment of $|f_p(t)-f_p(s)|$ we can cover the whole range $Q^{-1} \leq |t-s|\leq Q^{-3/8-\ep}$ using the Burgess bound \eqref{Eq.Burgess}.
\end{rem}

 \subsection{The range $Q^{-2/5}\leq |t-s|\leq 1$} This final range is the hardest to deal with. Indeed, a new feature occurs here (and more generally in the range $Q^{-1/2+\varepsilon} \leq |t-s|\leq 1$), caused by the fact that the random variables $\X_n$ are not independent, and $F_{\X,p}$ are not sub-Gaussian. In this case we shall use the quadratic large sieve to prove the following key proposition. 

\begin{pro}\label{ThirdRangeTSProposition}
Let $\beta=1/1000$ and $0\leq s<t\leq 1$ be real numbers such that $Q^{-2/5}\leq |t-s|\leq 1$. Then we have
$$
 \frac{1}{\pi^*(Q)}\sum_{Q\leq p\leq 2Q} |f_p(t)-f_p(s)|^{4} \ll |t-s|^{1+\beta}.
$$

\end{pro}
We shall require the following important large sieve inequalities for quadratic characters. The first is a direct consequence of a classical result due to Heath-Brown \cite{HB}.
\begin{lem}\label{HBLargeSieve}
Let $Q, Y\geq 2$. Then for arbitrary complex numbers $a_n$ 
and for any $\ep>0$ we have
$$
\sum_{Q\leq p\leq 2 Q} \left|\sum_{n\leq Y} a_n \chi_p(n)\right|^2 \ll_{\ep} (QY)^{\ep}(Q+Y) \sum_{\substack{m, n \leq Y\\ mn= \square}} |a_ma_n|.
$$
\end{lem}
\begin{proof}
This follows from Corollary 2 of \cite{HB} by embedding the set of primes $Q\leq p\leq 2Q$ in the set of fundamental discriminants $d$ with $|d|\leq 2Q$.
\end{proof}


As a corollary we deduce the following lemma. 
\begin{lem}\label{BoundTailPolyaRC} Let $\ep>0$ be a fixed small real number and $(a_n)_{n\geq 1}$ be an arbitrary bounded sequence of complex numbers. Let $Q$ be large, $1\leq Y\leq Q$ and $\delta>0$ be real numbers. The number of primes $Q\leq p\leq 2Q$ such that 
\begin{equation}\label{LargeTailRC}
\Bigg|\sum_{Y\leq |n|\leq Q} \frac{a_n\chi_p(n)}{n}\Bigg| >\delta,
\end{equation}
is 
$$\ll_{\ep} \frac{Q^{1+\ep}}{\delta^2 Y}.$$

\end{lem}
\begin{proof}

By Lemma \ref{HBLargeSieve} we have 
\begin{align*}
\sum_{Q\leq p\leq 2Q}\left|\sum_{Y\leq |n|\leq Q} \frac{a_n\chi_p(n) }{n}\right|^2 &\ll_{\ep} Q^{1+\ep/2} \sum_{\substack{Y\leq |n_1|, |n_2|\leq Q\\ |n_1n_2|=\square}} \left|\frac{a_{n_1}a_{n_2}}{n_1n_2}\right|
 \ll_{\ep} Q^{1+\ep/2} \sum_{n>Y} \frac{d(n^2)}{n^2},
\end{align*}
by writing $n^2=|n_1n_2|$ and using that $(a_n)_{n\geq 1}$ is bounded.  Finally, using the trivial bound $d(n^2)\ll_{\ep} n^{\ep/2}$ we deduce that the number of primes $Q\leq p\leq 2Q$ such that \eqref{LargeTailRC} holds is 
$$ \ll_{\ep} \frac{1}{\delta^2} Q^{1+\ep/2} \sum_{n>Y} \frac{d(n^2)}{n^2}
\ll_{\ep} \frac{1}{\delta^2} Q^{1+\ep/2} \sum_{n>Y} \frac{1}{n^{2-\ep/2}} \ll_{\ep} \frac{Q^{1+\ep/2}}{\delta^2 Y^{1-\ep/2}}\ll_{\ep} \frac{Q^{1+\ep}}{\delta^2 Y},
 $$
 as desired.

\end{proof} 

The next lemma is a large sieve inequality for prime discriminants, which is a special case of Lemma 9 of Montgomery and Vaughan \cite{MoVa79} (see also Lemma 1 of \cite{BaMo}).

\begin{lem} [Lemma 9 of \cite{MoVa79}]\label{LargeSieveMV}
Let $\ep>0$ be a fixed small number. Let $x, N$ be real numbers such that $x\geq 2$ and  $2\leq N\leq x^{1/2-\ep}$. Then for arbitrary complex numbers $a_1, \dots, a_N$ we have 
$$
\sum_{p\le x} \left|\sum_{n\leq N} a_n \chi_p(n)\right|^2 \ll \frac{x}{\log x} \sum_{\substack{m, n \leq N\\ mn= \square}} |a_ma_n|. 
$$
\end{lem}

\begin{proof}[Proof of Proposition \ref{ThirdRangeTSProposition}]

 Taking $Z=Q$ in P\'olya's Fourier expansion \eqref{eq: definition Legendre path} gives for all primes $Q\leq p\leq 2Q$ and all $t\in [0, 1]$
\begin{equation}\label{Eq.PolyaFinal}
  f_p(t)= \frac{\tau(\chi_p)}{2\pi i\sqrt{p}}\sum_{1\leq |n|\leq Q} \frac{\chi_p(n) (1-e(-nt))}{n} +O\left(\frac{\log Q}{\sqrt{Q}}\right).  
\end{equation}
Let $0\leq s<t\leq 1$ be real numbers such that $Q^{-2/5}\leq |t-s|\leq 1$. Using \eqref{Eq.PolyaFinal} together with the easy inequality $|a+b|^{4} \leq 2^{4}(|a|^{4}+|b|^{4})$, valid for all $a, b\in \mathbb{R}$, we get
\begin{align}\label{InequalityI}
\sum_{Q\leq p\leq 2Q} |f_p(t)-f_p(s)|^4&= \sum_{Q\leq p\leq 2Q} \left|\frac{\tau(\chi_p)}{2\pi i\sqrt{p}}\sum_{1\leq |n|\leq Q} \frac{\chi_p(n) (e(-ns)-e(-nt))}{n} +O\left(\frac{\log Q}{\sqrt{Q}}\right)\right|^4 \nonumber\\
&\ll \sum_{Q\leq p\leq 2Q} \left|\sum_{1\leq |n|\leq Q} \frac{\chi_p(n) (e(-ns)-e(-nt))}{n} \right|^4+ O\left(\frac{(\log Q)^3}{Q}\right).
\end{align}
Therefore, our goal is to show that 
$$ \mathcal{I}:= \sum_{Q\leq p\leq 2Q} \left|\sum_{1\leq |n|\leq Q} \frac{\chi_p(n) (e(-ns)-e(-nt))}{n} \right|^4\ll \pi^*(Q) |t-s|^{1+\beta}.$$ Hence, throughout the proof we will work under the assumption that
\begin{equation}\label{Eq.AssumptionI}
    \mathcal{I}\geq \pi^*(Q) |t-s|^{1+\beta},
\end{equation}  otherwise there is nothing to prove. 
Let $\ep>0$ be a small fixed real number. Let  $0<\delta<1$ and $1\leq Y\leq Q$ be parameters to be chosen, and define $\mathcal{A}(\delta, Y)$ to be the set of primes $p\in [Q, 2Q]$ such that 
$$ \left|\sum_{Y\leq |n|\leq Q} \frac{\chi_p(n) (e(-ns)-e(-nt))}{n}\right| >\delta.$$
Then it follows from Lemma \ref{BoundTailPolyaRC} (with the choice $a_n=e(-ns)-e(-nt)$)
that 
\begin{equation}\label{ConsequenceLemma4.5}
    |\mathcal{A}(\delta, Y)|\ll_{\ep} \frac{Q^{1+\ep/2}}{\delta^2 Y}.
\end{equation}
We now write 
$$ \mathcal{I}:= \mathcal{I}_1(\delta, Y)+\mathcal{I}_2 (\delta, Y),
$$
where $\mathcal{I}_1(\delta, Y)$ denotes the corresponding sum over primes $p\in [Q, 2Q]\setminus\mathcal{A}(\delta, Y)$ and $\mathcal{I}_2(\delta, Y)$ denotes the sum over those in $\mathcal{A}(\delta, Y)$. We start by bounding the latter term. By the Cauchy-Schwarz inequality and the bound \eqref{ConsequenceLemma4.5} we derive 
\begin{equation}\label{BoundI2}
\begin{aligned}
\mathcal{I}_2(\delta, Y) & \leq |\mathcal{A}(\delta, Y)|^{1/2}\left(\sum_{p \in \mathcal{A}(\delta, Y)} \left|\sum_{1\leq |n|\leq Q} \frac{\chi_p(n) (e(-ns)-e(-nt))}{n} \right|^8\right)^{1/2}\\
& \ll_{\ep} \frac{Q^{1/2+\ep/4}}{\delta Y^{1/2}} \left((\log Q)^4\sum_{p \in \mathcal{A}(\delta,Y)} \left|\sum_{1\leq |n|\leq Q} \frac{\chi_p(n) (e(-ns)-e(-nt))}{n} \right|^4\right)^{1/2}\\
& \ll_{\ep} \frac{Q^{1/2+\ep/2}}{\delta Y^{1/2}} \sqrt{\mathcal{I}}\ll_{\ep} \frac{Q^{\ep}}{\delta Y^{1/2} |t-s|^{(1+\beta)/2}} \mathcal{I}.
\end{aligned}
\end{equation}
by \eqref{Eq.AssumptionI} and since the inner sum over $n$ is trivially $\ll \log Q$. 

We now handle $\mathcal{I}_1(\delta, Y)$. Using our assumption on $\mathcal{A}(\delta,Y)$ together with the inequality $|a+b|^{4} \leq 2^{4}(|a|^{4}+|b|^{4})$ we obtain 
\begin{align*}
\mathcal{I}_1(\delta, Y)&= \sum_{p\in [Q, 2Q]\setminus\mathcal{A}(\delta,Y)} \left|\sum_{1\leq |n|\leq Y} \frac{\chi_p(n) (e(-ns)-e(-nt))}{n} +O(\delta) \right|^4\\
& \ll \sum_{Q\leq p\leq 2Q} \left|\sum_{1\leq |n|\leq Y} \frac{\chi_p(n) (e(-ns)-e(-nt))}{n} \right|^4 + \pi^{*}(Q)\delta^4.
\end{align*}
We choose $\delta:= |t-s|^{(1+\beta)/4}$. With this choice, the previous  estimate and \eqref{BoundI2} become
\begin{equation}\label{Eq.SecondBoundI1}
\mathcal{I}_1(\delta, Y) \ll \sum_{Q\leq p\leq 2Q} \left|\sum_{1\leq |n|\leq Y} \frac{\chi_p(n) (e(-ns)-e(-nt))}{n} \right|^4 + \pi^{*}(Q)|t-s|^{1+\beta},
\end{equation}
and 
\begin{equation}\label{Eq.SecondBoundI2}
\mathcal{I}_2(\delta, Y)  \ll_{\ep} \frac{Q^{\ep}}{Y^{1/2} |t-s|^{3(1+\beta)/4}} \mathcal{I}.
\end{equation}

To complete the proof we will choose $Y$ in such a way to insure that $\mathcal{I}_2(\delta, Y)\leq \mathcal{I}/2$. This will imply $\mathcal{I}\leq 2\mathcal{I}_1(\delta, Y)$, and hence it will only remain to bound the fourth moment on the right hand side of \eqref{Eq.SecondBoundI1}, which we shall handle using the quadratic large sieve. The choice of the parameter $Y$ as well as which large sieve inequality to use (Heath-Brown's large sieve (Lemma \ref{HBLargeSieve}), or the Montgomery-Vaughan large sieve (Lemma \ref{LargeSieveMV})) will depend on the size of $|t-s|$. We will consider two cases:
\subsection*{Case 1: $Q^{-2/5}\leq |t-s|\leq Q^{-1/6}$}
In this range we use Heath-Brown's large sieve. Indeed,  
 by Lemma \ref{HBLargeSieve} and Eq. \eqref{BoundsGSum} we obtain
\begin{align}\label{BoundI12}
\sum_{Q\leq p\leq 2Q} \left|\sum_{1\leq |n|\leq Y} \frac{\chi_p(n) (e(-ns)-e(-nt))}{n} \right|^4
&= \sum_{Q\leq p\leq 2Q} \left|\sum_{1\leq |m|\leq Y^2} \frac{\chi_p(m)g(m)}{m} \right|^2 \nonumber \\
&\ll_{\ep} Q^{\ep/2} (Q+Y^2) \sum_{\substack{1\leq |n_1|, |n_2|\leq Y^2\\ n_1n_2=\square}} \left|\frac{g(n_1)g(n_2)}{n_1n_2}\right| \\
&\ll_{\ep} Q^{\ep/2} (Q+Y^2)  |t-s|^{2-2\ep} \nonumber,
\end{align}
where $g$ is defined by \eqref{Definitiong}.
We now choose $Y=Q^{1/2-\ep}|t-s|^{(\beta-1)/2}$. Then by \eqref{Eq.SecondBoundI2} we obtain
$$ \mathcal{I}_2(\delta, Y)  \ll_{\ep} \frac{Q^{-1/4+2\ep}} {|t-s|^{1/2+\beta}} \mathcal{I} \leq \frac{\mathcal{I}}{2}, $$
if $\ep$ is suitably small and $Q$ is large enough, by our assumption on $|t-s|$.
Therefore, combining the estimates  \eqref{Eq.SecondBoundI1} and \eqref{BoundI12}  we deduce that 
\begin{equation}\label{BoundI0}
\mathcal{I}\leq 2 \mathcal{I}_1(\delta, Y)\ll_{\ep} Q^{\ep/2} (Q+Y^2)  |t-s|^{2-2\ep} + \pi^*(Q)|t-s|^{1+\beta} \ll\pi^*(Q)|t-s|^{1+\beta},
\end{equation}
as desired. 
\subsection*{Case 2: $Q^{-1/6}\leq|t-s|\leq 1$}
In this range we can no longer use Lemma \ref{HBLargeSieve} (Heath-Brown's quadratic large sieve) to bound the left hand side of \eqref{BoundI12} since the extra factor $Q^{\ep/2}$ is problematic if $|t-s|$ is large (for example if $|t-s|\asymp 1$). In this case we choose $Y= Q^{1/2-\ep}$ and use the large sieve inequality of Montgomery and Vaughan (Lemma \ref{LargeSieveMV}) instead. Indeed, it follows from Lemma \ref{LargeSieveMV} and Eq. \eqref{BoundsGSum} that 
\begin{equation}\label{BoundI13}
\begin{aligned}
\sum_{Q\leq p\leq 2Q} \left|\sum_{1\leq |n|\leq Y} \frac{\chi_p(n) (e(-ns)-e(-nt))}{n} \right|^4
&= \sum_{Q\leq p\leq 2Q} \left|\sum_{1\leq |m|\leq Y^2} \frac{\chi_p(m)g(m)}{m} \right|^2 \\
&\ll \pi^*(Q) \sum_{\substack{1\leq |n_1|, |n_2|\leq Y^2\\ n_1n_2=\square}} \left|\frac{g(n_1)g(n_2)}{n_1n_2}\right| \\
&\ll_{\ep}  \pi^*(Q)|t-s|^{2-2\ep}.
\end{aligned}
\end{equation}
Choosing $\ep$ to be suitably small and inserting this estimate in \eqref{Eq.SecondBoundI1} gives 
$$ \mathcal{I}_1(\delta, Y) \ll  \pi^{*}(Q)|t-s|^{1+\beta}. $$
Finally, using \eqref{Eq.SecondBoundI2} and our assumption on $|t-s|$ we derive
$$ \mathcal{I}_2(\delta, Y)  \ll_{\ep} \frac{Q^{-1/4+2\ep}} {|t-s|^{3(1+\beta)/4}} \mathcal{I} \leq \frac{\mathcal{I}}{2}, $$
if $\ep$ is suitably small and $Q$ is large enough. Combining these estimates we deduce that 
$$\mathcal{I}\leq 2 \mathcal{I}_1(\delta, Y) \ll  \pi^{*}(Q)|t-s|^{1+\beta}, $$
which completes the proof.
\begin{rem}\label{Rem.FinalRange}
By optimizing this method and allowing the exponent $\beta>0$ to be very close to $0$, we can cover the wider range $Q^{-1/2+\ep} \leq |t-s|\leq 1$.
\end{rem}

\end{proof}

\end{document}